\documentclass [11pt]{amsart}  
\usepackage{a4}                 
\usepackage{paralist}           
\usepackage{graphicx}
\usepackage{amssymb}                               
\usepackage{float}
\usepackage{amsmath}
\usepackage{psfrag}
\usepackage{MnSymbol}

\usepackage{amsmath,amsthm}

\theoremstyle{plain}
\newtheorem{thm}{Theorem}[section]

\newtheorem{lem}[thm]{Lemma}
\newtheorem{prop}[thm]{Proposition}
\newtheorem{cor}[thm]{Corollary}

\theoremstyle{definition}
\newtheorem{defn}[thm]{Definition}

\theoremstyle{remark}
\newtheorem*{rem}{Remark}
\newtheorem*{note}{Note}

\newtheoremstyle{TheoremNum}
        {\topsep}{\topsep}              
        {\itshape}                      
        {}                              
        {\bfseries}                     
        {.}                             
        { }                             
        {\thmname{#1}\thmnote{ \bfseries #3}}
    \theoremstyle{TheoremNum}
    \newtheorem{thmn}{Theorem}

\newtheorem{lemn}{Lemma}
\newtheorem{propn}{Proposition}

\DeclareMathOperator{\id}{Id}
\DeclareMathOperator{\image}{Im}

\DeclareMathOperator{\BS}{BS}
\DeclareMathOperator{\ths}{th}
\DeclareMathOperator{\ti}{T}

\title{The subgroup identification problem for finitely presented groups}
\author{Maurice Chiodo}
\date{\today}

\begin{document}

\let\thefootnote\relax\footnotetext{2010 \textit{AMS Classification:} 20F10, 03D40, 03D80.}
\let\thefootnote\relax
\let\thefootnote\relax\footnotetext{\textit{The author was supported by:}}
\let\thefootnote\relax\footnotetext{\emph{An} Australian Postgraduate Award, \emph{and a} David Hay Postgraduate Writing-Up Award.}

\begin{abstract}
We introduce the subgroup identification problem, and show that there is a finitely presented group G for which it is unsolvable, and that it is uniformly solvable in the class of finitely presented locally Hopfian groups.  This is done as an investigation into the difference between strong and weak effective coherence for finitely presented groups.
\end{abstract}

\maketitle

\section{Introduction}

Decision problems for finitely presented groups are well-studied, going back to the work of Dehn \cite{Dehn} in 1911. The \emph{word problem}, of deciding when a word represents the trivial element in a group, was the first of these to be well understood, with the independent work of Boone, Britton and Novikov \cite{Boone, Britton, Novikov}. A clever adaption of this problem by the independent work of Adian and Rabin \cite{Adian, Rabin} gives rise to the impossibility of the \emph{isomorphism problem}, of deciding if two finite presentations describe isomorphic groups. To list all major results in group-theoretic decision problems would be a substantial undertaking, and we defer to \cite{Mille-92} for an excellent introduction and survey of the discipline up to 1992. 

For the remainder of this section we write $X^{*}$ to denote all finite words on a set $X$, $\overline{P}$ to denote the group given by a presentation $P$, and $\langle w_{1}, \ldots, w_{n} \rangle^{\overline{P}}$ to denote the subgroup in $\overline{P}$ generated by words $ w_{1}, \ldots, w_{n}$ (see section \ref{notation} for a clarification of this notation if required).

We introduce the following problem, and give the main result of this work.

\begin{defn}\label{sub ID}
We say that the \emph{subgroup identification problem} for a finite presentation $P=\langle X | R \rangle$ is solvable if there is a partial algorithm that, on input of any finite set of words $ w_{1}, \ldots, w_{n} \in X^{*}$, and any finite presentation $Q=\langle Y | S \rangle$, outputs (if the subgroup $\langle w_{1}, \ldots, w_{n} \rangle^{\overline{P}} \cong \overline{Q}$) an explicit set map $\phi$ from $Y$ to $\{w_{1}, \ldots, w_{n}\}^{*}$ which extends to an isomorphism $\overline{\phi}: \overline{Q} \to \langle w_{1}, \ldots, w_{n} \rangle^{\overline{P}}$. That is, the induced homomorphism $\overline{\phi}: \overline{Q} \to \overline{P}$ is injective, and $\image(\overline{\phi})=\langle w_{1}, \ldots, w_{n} \rangle^{\overline{P}}$. Otherwise, we place no conditions on the output, or even whether it halts.
\end{defn}

\begin{thmn}[\ref{main}]
There is a finitely presented group with unsolvable subgroup identification problem.
\end{thmn}
\mbox{}

As a consequence of lemma \ref{isom check}, having solvable subgroup identification problem depends only on the isomorphism class of a finite presentation.

In \cite[Definition 1.1]{GroWil} a finitely generated group $G$ is said to be \emph{effectively coherent} if it is \emph{coherent} (all of its finitely generated subgroups are finitely presentable) and there is an algorithm for $G$ that, on input of a finite collection of words $S$, outputs a finite presentation for the subgroup generated by $S$. This motivates us to make the following two definitions:

\begin{defn}\label{weak eff coh}
We say a coherent group $G=\overline{\langle X|R \rangle}$ is \emph{weakly effectively coherent} if there is an algorithm for $G$ that, on input of a finite collection of words $S \subset X^{*}$, outputs a finite presentation for the subgroup $\langle S \rangle^{G}$ generated by $S$.
\end{defn}

\begin{defn}\label{strong eff coh}
We say a coherent group $G=\overline{\langle X|R \rangle}$ is \emph{strongly effectively coherent} if there is an algorithm for $G$ that, on input of a finite collection of words $S \subset X^{*}$, outputs a finite presentation $P=\langle Y|V\rangle$ for the subgroup $\langle S \rangle^{G}$ generated by $S$, along with a map $\phi: Y \to X^{*}$ which extends to an injection $\overline{\phi}: \overline{P} \to G$ whose image is $\langle S \rangle^{G}$.
\end{defn}

By definition, strongly effectively coherent groups are weakly effectively coherent. Conversely, 
it is imediate that the notions of strong and weak effective coherence are equivalent in groups with solvable subgroup identification problem. A motivating (unresolved) question for this work is whether strong and weak effective coherence are equivalent notions for all finitely presented groups.

The subgroup identification problem was formulated after the author read the proof of \cite[Lemma 1.2]{GroWil} and \cite[Remark 1.3]{GroWil} where, were it not for a later typographical correction provided in \cite[Definition 1.2]{GroWilcor}, it would suggest that the subgroup identification problem is (uniformly) solvable over all finitely presented groups. This would, in turn, imply that weak and strong effective coherence are equivalent notions. The groups encountered in \cite{GroWil} all have solvable subgroup identification problem, as they are all locally Hopfian (see theorem \ref{can do for locally hopf} below). In addition, for all results shown in \cite{GroWil} all weakly effectively coherent groups are actually proven to be strongly effectively coherent, so combined with the typographical correction from \cite[Definition 1.2]{GroWilcor} this then becomes a moot point. We do not suggest that there are any errors with the main conclusions of \cite{GroWil}. However, \cite[Lemma 1.2 and Remark 1.3]{GroWil} do raise an interesting question regarding the connection between strong and weak effective coherence.

An application of standard techniques gives the following result, which should be read very carefully as one could misinterpret it as showing that the subgroup identification problem is uniformly solvable for all finitely presented groups.
\\

\begin{propn}[\ref{almost}]
There is a uniform partial algorithm that, on input of: A finite presentation $P=\langle X | R \rangle$ of a group, and a finite set of words $w_{1}, \ldots, w_{n} \in X^{*} $ such that $\langle w_{1}, \ldots, w_{n} \rangle^{\overline{P}}$ is finitely presentable, and $Q=\langle Y | S \rangle$ a finite presentation such that $\overline{Q}\cong \langle w_{1}, \ldots, w_{n}\rangle ^{\overline{P}}$;
\\Outputs: A finite set of words $c_{1}, \ldots, c_{k} \in \{w_{1}, \ldots, w_{n}\}^{*}$ such that each $c_{i}$ is trivial in $\overline{P}$ (when viewed as a word in $X^{*}$) and $\overline{\langle w_{1}, \ldots, w_{n} | c_{1}, \ldots, c_{k} \rangle}$ is isomorphic to $\overline{Q}$ and hence also to $\langle w_{1}, \ldots, w_{n} \rangle ^{\overline{P}}$.
\end{propn}
\mbox{}

A group $G$ is \emph{Hopfian} if every surjective endomorphism of $G$ is  an isomorphism, and \emph{locally Hopfian} if every finitely generated subgroup is Hopfian.
\\

\begin{thmn}[\ref{can do for locally hopf}]
The subgroup identification problem is uniformly solvable in the class of finitely presented locally Hopfian groups. Hence, in this class of groups, weak effective coherence is equivalent to strong effective coherence.
\end{thmn}
\mbox{}

Thus the notions of weak and strong effective coherence are equivalent in the class of locally Hopfian groups, and are so because in this class the subgroup identification problem is uniformly solvable. However, as soon as we start dealing with non-Hopfian groups, the situation is a lot different.

The following theorem, and subsequent lemma, are vital to our main result. They show that a solution to the word problem for a finite presentation cannot be uniformly lifted to a solution to the word problem for a recursive presentation isomorphic to that same group (so an isomosphism between the two cannot be constructed). This contrasts the case of having a pair of isomorphic finite presentations, where an isomorphism can be uniformly constructed (lemma \ref{isom check}).
\\

\begin{thmn}[\ref{thm pre main}]
There is a finite presentation $P$ of a group with solvable word problem (namely, the Baumslag-Solitar group $\BS(2,3)$) for which there is no partial algorithm that, on input of a recursive presentation $Q$ such that $\overline{P} \cong \overline{Q}$, outputs a solution to the word problem for $Q$.
\end{thmn}
\mbox{}

\begin{lemn}[\ref{pre main}]
There is a finite presentation $P=\langle X|R \rangle$ of a group with solvable word problem (namely, $\BS(2,3)$) for which there is no partial algorithm that, on input of a recursive presentation $Q=\langle Y|S \rangle$ such that $\overline{P} \cong \overline{Q}$, outputs a set map $\phi: X \to Y^{*}$ which extends to an isomorphism $\overline{\phi} : \overline{P} \to \overline{Q}$.
\end{lemn}
\mbox{}

By viewing isomorphisms between groups as Tietze transformations rather than set maps (see \cite[$\S 1.5$]{MKS} for a full exposition of this concept), we can re-interpret lemma \ref{pre main} in the following way.
\\

\begin{lemn}[\ref{Ti version}]
There is a finite presentation $P$ of a group with solvable word problem (namely, $\BS(2,3)$) for which there is no algorithm that, on input of a recursive presentation $Q$ such that $\overline{P} \cong \overline{Q}$, outputs a recursive enumeration of Tietze transformations of type $(\ti 1)$ and $(\ti 2)$ (manipulation of relators), and a finite set of Tietze transformations of type $(\ti 3)$ and $(\ti 4)$ (addition and removal of generators), with notation as per \cite[$\S 1.5$]{MKS}, transforming $P$ to $Q$.
\end{lemn}
\mbox{}

A direct consequence of the Higman embedding theorem \cite[Theorem 1]{Hig emb} is that any recursively presented group can be uniformly embedded into a finitely presented group. By a careful application of this result, paying special attention to the uniformity in which such a finite presentation is constructed as given in the proof of \cite[Theorem 12.18]{Rot}, our main result (theorem \ref{main}) follows. 
\\

\noindent \textbf{Acknowledgements}: The author would like to thank Jack Button and Andrew Glass for their general advice, and Chuck Miller for giving general comments and corrections as well as pointing out a more elegant proof of lemma \ref{jack} and a much more direct proof of theorem \ref{main}. Thanks also go to Nicholas Touikan for his helpful suggestions for a revised version of this work, and to Daniel Groves and Henry Wilton for explaining their work in \cite{GroWil}.

\section{Preliminaries}

\subsection{Standard notation}\label{notation}

With the convention that $\mathbb{N}$ contains $0$, we take $\varphi_{m}$ to be the $m^{\ths}$ \emph{partial recursive function} $\varphi_{m}: \mathbb{N} \to \mathbb{N}$, and the $m^{\ths}$ \emph{partial recursive set} (r.e.~set) $W_{m}$ as the domain of $\varphi_{m}$.
If $P=\langle X|R\rangle$ is a group presentation with generating set $X$ and relators $R$, then we denote by $\overline{P}$ the group presented by $P$. A presentation $P=\langle X|R\rangle$ is said to be a \emph{recursive presentation} if $X$ is a finite set and $R$ is a recursive enumeration of relators; $P$ is said to be an \emph{infinite recursive presentation} if instead $X$ is a recursive enumeration of generators. A group $G$ is said to be \emph{finitely presentable} if $G\cong \overline{P}$ for some finite presentation $P$.
If $P,Q$ are group presentations then we denote their free product presentation by $P*Q$, given by taking the disjoint union of their generators and relators; this extends to the free product of arbitrary collections of presentations.
If $X$ is a set, then we denote by $X^{-1}$ a set of the same cardinality as $X$ along with a fixed bijection $\phi: X \to X^{-1}$, where we denote $x^{-1}:=\phi(x)$. We write $X^{*}$ for the set of finite words on $X \cup X^{-1}$, including the empty word $\emptyset$. If $g_{1}, \ldots, g_{n}$ are a collection of elements of a group $G$, then we write $\langle   g_{1}, \ldots, g_{n} \rangle^{G}$ for the subgroup in $G$ generated by these elements, and $\llangle g_{1}, \ldots, g_{n} \rrangle^{G}$ for the normal closure of these elements in $G$. A group $G$ is \emph{Hopfian} if every surjective endomorphism of $G$ is  an isomorphism, and \emph{locally Hopfian} if every finitely generated subgroup is Hopfian. Finally, the commutator $[x,y]$ is taken to be $xyx^{-1}y^{-1}$.

\subsection{Groups}

It is a result by Mihailova \cite{Mihai} that there exists a finitely presented group with solvable word problem for which the \emph{subgroup membership problem}, of deciding when a word on the generators lies in a given finitely generated subgroup, is unsolvable. We mention this in the context of the following known result about the word problem for HNN extensions of groups.

\begin{lem}\label{swp}
Let $H, K \leqslant G$ be isomorphic finitely generated subgroups of $G$, each having solvable subgroup membership problem in $G$. Let $\varphi:H \to K$ be an isomorphism. Then the HNN extension $G*_{\varphi}$ has solvable word problem.
\end{lem}

\begin{proof}
This is immediate from the normal form theorem for HNN extensions (see \cite[$\S$IV Theorem 2.1]{LynSch}).
\end{proof}

The following group construction from \cite{BaumSol} plays an important part in our arguments, as it is an explicit example of a finitely presented non-Hopfian group with solvable word problem.

\begin{defn}\label{BS groups}
The \textit{Baumslag-Solitar groups} $\overline{\BS}(m,n)$ are defined via the following finite presentations, each an HNN extension of $\mathbb{Z}$:
\[
\BS(m,n):= \langle s, t | s^{-1}t^{m}s=t^{n} \rangle 
\]
\end{defn}

\begin{thm}[{Baumslag-Solitar \cite[Theorem 1]{BaumSol}}]\label{BS}
The group $\overline{\BS}(2,3)$ is non-Hopfian and has solvable word problem. The map $\overline{f}: \overline{\BS}(2,3) \to \overline{\BS}(2,3)$ given by extending the map $f: \{s,t\} \to \{s,t\}^{*}$, $f(s)= s$, $f(t)=t^{2}$, is a non-injective epimorphism. The word $[s^{-1}ts, t]$ is non-trivial in $\overline{\BS}(2,3)$ and lies in $\ker(\overline{f})$.
\end{thm}

\begin{proof}
That $\overline{\BS}(2,3)$ has solvable word problem comes from the fact that it is an HNN extension of $\mathbb{Z}$ (see lemma \ref{swp}, or alternatively a result by Magnus in \cite[$\S$ IV Theorem 5.3]{LynSch}  that all 1-relator groups have solvable word problem). The remainder of the theorem is proved in \cite[Theorem 1]{BaumSol}.
\end{proof}

\subsection{Enumerable processes in groups}

We note some partial algorithms and recursively enumerable sets, in the context of group presentations. These are standard results, which we state (without proof) for the convenience of any reader not familiar with the area.

\begin{lem}\label{words}
Let $P=\langle X | R \rangle$ be a recursive presentation. Then the words in $X^{*}$ which represent the identity in $\overline{P}$ are recursively enumerable. Moreover, this algorithm is uniform over all recursive presentations.
\end{lem}

\begin{lem}\label{hom check}
There is a partial algorithm that, on input of two finite presentations $P=\langle X | R \rangle$ and $Q=\langle Y | S \rangle$, and a set map $\phi: X \to Y^{*}$, halts if and only if $\phi$ extends to a homomorphism $\overline{\phi}: \overline{P} \to \overline{Q}$.
\end{lem}

\begin{lem}\label{isom check}
There is a partial algorithm that, on input of two finite presentations $P=\langle X | R \rangle$ and $Q=\langle Y | S \rangle$, halts if and only if $\overline{P} \cong \overline{Q}$, and outputs an isomorphism between them.
\end{lem}

It is important to note that lemma \ref{isom check} does not hold if we instead consider recursive presentations. In fact, even if we start with one recursive presentation, and one finite presentation, the lemma does not hold; we show this later as lemma \ref{pre main}.

\begin{lem}
Let $P=\langle X | R \rangle$ be a finite presentation. Then the set of finite presentations defining groups isomorphic to $\overline{P}$ is recursively enumerable. Moreover, this enumeration algorithm is uniform over all finite presentations.
\end{lem}

\section{Initial questions and observations}\label{here}

\subsection{The finite presentation problem}

Following \cite[Definition 1.1]{BriWil}, we say that the \emph{finite presentation problem} for a finitely presented group $\overline{\langle X | R \rangle}$ is solvable if there is a partial algorithm that, on input of a finite set of words $w_{1}, \ldots, w_{n} \in X^{*}$ such that $\langle w_{1}, \ldots, w_{n} \rangle ^{\overline{P}}$ is finitely presentable, outputs a finite presentation $Q$ for $\langle w_{1}, \ldots, w_{n} \rangle ^{\overline{P}}$. By definition, a finitely presented weakly effectively coherent group has solvable finite presentation problem. The existence of a finitely presented group with unsolvable word problem immediately gives rise to the following obsevation:

\begin{lem}\label{fp prob}
There is a finite presentation $P=\langle X | R \rangle$ of a group for which the finite presentation problem is unsolvable.
\end{lem}

\begin{proof}
Take $P=\langle X | R \rangle$ to be a finite presentation of a group with unsolvable word problem (see \cite[Lemma 12.7]{Rot}), and suppose this has an algorithm to solve the finite presentation problem. Note that for any word $w \in X^{*}$, we have that $\langle w \rangle ^{\overline{P}}$ is cyclic, and hence finitely presentable. So, on input of a single word $w$, the algorithm will output a finite presentation $Q$ of the cyclic group $\langle w \rangle ^{\overline{P}}$. But the isomorphism problem for finitely presented abelian groups is solvable (see \cite{Mille-92} p.~31), so we can decide if $\overline{Q}$ is trivial, and hence if $w$ is trivial in $\overline{P}$, which is impossible since $\overline{P}$ has unsolvable word problem.
\end{proof}

\begin{rem}
As pointed out to the author by Chuck Miller, the work of Collins in \cite{Collins} shows that the group in lemma \ref{fp prob} can be taken to have solvable word problem. This is because such an algorithm for the finite presentation problem would explicitly solve the \emph{order problem}, of deciding what the order of an element of the group is, and a direct consequence of \cite[Theorem A]{Collins} is that there is a finitely presented group with solvable word problem and unsolvable order problem.
\end{rem}

Bridson and Wilton \cite{BriWil} give an in-depth analysis of the finite presentation problem, showing that it is not uniformly solvable in several classes of finitely presented groups \cite[Corollary C]{BriWil}. Moreover, they construct a finite presentation of a group with polynomial Dehn function (and hence solvable word problem) which has unsolvable finite presentation problem \cite[Theorem E]{BriWil}.

Having seen that the finite presentation problem is unsolvable in general, we shift our attention to similar  questions. By considering the trivial group, and the fact that the triviality problem is undecidable, one can show that there is no algorithm that, given two finite presentations $P,Q$, determines if $\overline{P}$ embeds in $\overline{Q}$ or not. The following two stronger results were obtained in \cite{Chiodo}, and are closely related to the subgroup identification problem.

\begin{thm}[{Chiodo \cite[Theorem 6.8]{Chiodo}}]
There is a finitely presented group $G$ such that the set of finite presentations of groups which embed into $G$ is not recursively enumerable.
\end{thm}

\begin{thm}[{Chiodo \cite[Theorem 6.6]{Chiodo}}]\label{chi1}
There is no algorithm that, on input of two finite presentations $P=\langle X|R\rangle$, $Q=\langle Y|S\rangle$ such that $\overline{P}$ embeds in $\overline{Q}$, outputs an explicit map $\phi: X \to Y^{*}$ which extends to an embedding $\overline{\phi}: \overline{P} \hookrightarrow \overline{Q}$.
\end{thm}

In the proof of theorem \ref{chi1}, as found in \cite{Chiodo}, we see that the algorithmic problem arises from not definitely knowing a set of target words for an injection from $\overline{P}$ into $\overline{Q}$. Knowing which elements $\overline{P}$ maps to in $\overline{Q}$ brings us precisely to the subgroup identification problem.

\subsection{The subgroup identification problem for Hopfian groups}

The following is a standard result about finitely presentable groups.

\begin{lem}\label{truncate}
Let $ \langle X|R \rangle$ be a recursive presentation of a finitely presentable group. Then there is a finite subset $R'\subseteq R$ such that $\overline{\langle X|R \rangle} \cong \overline{\langle X|R' \rangle}$ via extention of the identity map on $X$. That is, there is a finite truncation $R'$ of $R$ such that all other relations are a consequence of the first $R'$.
\end{lem}

The following result could lead us to think that all finitely presented groups have solvable subgroup identification problem, and we urge the reader to study this carefully to become convinced that it is not what is proved.

\begin{prop}\label{almost}
There is a uniform partial algorithm that, on input of: A finite presentation $P=\langle X | R \rangle$ of a group, and a finite set of words $w_{1}, \ldots, w_{n} \in X^{*} $ such that $\langle w_{1}, \ldots, w_{n} \rangle^{\overline{P}}$ is finitely presentable, and $Q=\langle Y | S \rangle$ a finite presentation such that $\overline{Q}\cong \langle w_{1}, \ldots, w_{n}\rangle ^{\overline{P}}$;
\\Outputs: A finite set of words $c_{1}, \ldots, c_{k} \in \{w_{1}, \ldots, w_{n}\}^{*}$ such that each $c_{i}$ is trivial in $\overline{P}$ (when viewed as a word in $X^{*}$) and $\overline{\langle w_{1}, \ldots, w_{n} | c_{1}, \ldots, c_{k} \rangle}$ is isomorphic to $\overline{Q}$ and hence also to $\langle w_{1}, \ldots, w_{n} \rangle ^{\overline{P}}$.
\end{prop}

\begin{proof}
Begin an enumeration $c_{1}, c_{2}, \ldots$ of all words in the $\{w_{1}, \ldots, w_{n}\}^{*}$ which are trivial in $\overline{P}$ (when viewed as words in $X^{*}$); this can be done by repeated application of lemma \ref{words}. Define the presentation $P_{l}:= \langle w_{1}, \ldots, w_{n}  | c_{1}, \ldots, c_{l}  \rangle$.
By lemma \ref{truncate}, as $\langle w_{1}, \ldots, w_{n} \rangle ^{\overline{P}}$ is finitely presentable, and $T:=\langle w_{1}, \ldots, w_{n}  | c_{1}, c_{2}, \ldots \rangle$ is a recursive presentation of a group isomorphic to $\langle w_{1}, \ldots, w_{n} \rangle ^{\overline{P}}$ via extension of the map $w_{i} \mapsto w_{i}$, there exists some finite $m$ such that $\overline{P}_{m} \cong \overline{T}$ (again via extension of the map $w_{i} \mapsto w_{i}$). That is, using lemma \ref{truncate} we can truncate the relations of $T$ at position $m$ (forming $P_{m}$) such that all successive relations are consequences of the first $m$ (note that selecting $m$ is not an algorithmic process; for the moment we merely appeal to the fact that such an $m$ exists). So we have $\overline{Q}\cong \langle w_{1}, \ldots, w_{n} \rangle ^{\overline{P}} \cong \overline{T}\cong \overline{P}_{m}$, where $Q$ is our given explicit finite presentation. Now we use lemma \ref{isom check} to begin checking for an isomorphism between $\overline{Q}$ and $\overline{P}_{1}, \overline{P}_{2}, \ldots$ as a parallel process. Eventually this process will stop at some $k$ (perhaps different from $m$) such that $\overline{P}_{k} \cong \overline{T} \cong \overline{Q}$.
\end{proof}

Note that we were very careful to mention that the selection of $m$ in the above proof was existential, but not necessarily recursive. This is very important, as later in corollary \ref{find m} we construct a class of recursive presentations for which the selection of such an $m$ is provably non-recursive.

\begin{rem}
With the above proof in mind, one would naively hope that the map from $ P_{k}= \langle w_{1}, \ldots, w_{n}  | c_{1}, \ldots, c_{k}  \rangle$ to $P =\langle X | R \rangle$ given by extending $w_{i} \mapsto w_{i} \in X^{*}$ would be an injection onto its image $\langle w_{1}, \ldots, w_{n} \rangle ^{\overline{P}}$; in the case that $\overline{P}_{k}$ is Hopfian this would indeed be true, as shown later in in lemma \ref{hopf answer}. But theorem \ref{main} shows that this is not true in general.
\end{rem}

The existence of a finitely presented non-Hopfian group, combined with lemma \ref{truncate}, immediately gives rise to the following:

\begin{lem}\label{jack}
There exist finite presentations $P=\langle x_{1}, \ldots, x_{n} | r_{1}, \ldots, r_{m} \rangle$ and $Q=\langle x_{1}, \ldots, x_{n} | r_{1}, \ldots, r_{m}, s_{1}, \ldots, s_{k} \rangle$, along with a recursive presentation $T=\langle x_{1}, \ldots, x_{n} | r_{1},r_{2}, \ldots \rangle$ (where $r_{i}$ are the same in $P, Q, T$ for all $i \leq m$) such that:
\\$1$. $\overline{P}\cong \overline{Q} \cong \overline{T}$ (and hence $\overline{T}$ is finitely presentable).
\\$2$. The quotient maps $\overline{\phi}: \overline{P} \to \overline{Q}$ and $\overline{\psi}: \overline{P} \to \overline{T}$, each an extension of the map $x_{i} \mapsto x_{i}$, are not isomorphisms.
\end{lem}

Observe that the above quotient maps are always isomorphisms in the case of Hopfian groups, by very definition. The point of making the above observation is to stress the following:
\\1. There may be many ways to truncate the relators of $T$ to get a presentation of a group isomorphic to $\overline{P}$.
\\2. Truncating $T$ to get a finite presentation $T'$ with $\overline{T'} \cong \overline{T}$ may not give a presentation which is isomorphic to $\overline{P}$ via the quotient map we have been discussing.

\begin{lem}\label{hopf answer}
Let $P$, $P_{k}$ and $\{w_{1}, \ldots, w_{n}\}$ be as in the proof of proposition $\ref{almost}$. If $\langle w_{1}, \ldots, w_{n} \rangle ^{\overline{P}}$ is Hopfian, then the map $\phi: \{w_{1}, \ldots, w_{n}\} \to X^{*}$ sending $w_{i}$ (as a generator of $P_{k}$) to $w_{i}$ (as a word in $X^{*}$) extends to a monomorphism $\overline{\phi}: \overline{P}_{k} \hookrightarrow \overline{P}$, and thus to an isomorphism to its image $\langle w_{1}, \ldots, w_{n} \rangle ^{\overline{P}}$.
\end{lem}

\begin{proof}
By the proof of proposition \ref{almost}, we know that $\phi$ extends to a surjection $\overline{\phi}: \overline{P}_{k} \twoheadrightarrow \langle w_{1}, \ldots, w_{n} \rangle ^{\overline{P}}$. But $\overline{P}_{k} \cong \langle w_{1}, \ldots, w_{n} \rangle ^{\overline{P}}$, which is Hopfian. Hence $\overline{\phi}$ must be injective, and thus an isomorphism.
\end{proof}

\begin{thm}\label{can do for locally hopf}
The subgroup identification problem is uniformly solvable in the class of finitely presented locally Hopfian groups. Hence, in this class of groups, weak effective coherence is equivalent to strong effective coherence.
\end{thm}

\begin{proof}
The first part follows from lemma \ref{hopf answer}; the second is then immediate.
\end{proof}

\section{Main results}

We begin by making the important observation that a solution to the word problem can't be algorithmically lifted from a finite presentation of $\overline{\BS}(2,3)$ to a recursive presentation of $\overline{\BS}(2,3)$.

\begin{thm}\label{thm pre main}
There is a finite presentation $P$ of a group with solvable word problem (namely, $\BS(2,3)$) for which there is no partial algorithm that, on input of a recursive presentation $Q$ such that $\overline{P} \cong \overline{Q}$, outputs a solution to the word problem for $Q$.
\end{thm}

\begin{proof}
Assume we have such an algorithm. Take $P=\langle X|R \rangle$ to be a finite presentation for $\overline{\BS}(2,3)$. Fix $Q= \langle X|R \cup S \rangle$ to be a finite presentation of a non-trivial quotient which is isomorphic to $\overline{P}$, but not by the extension $\overline{\id}_{X}$ of the map $\id_{X}: X \to X$ (say, instead, by the extension $\overline{\phi}$ of the map $\phi: X \to X^{*}$). Fix a word $w\in X^{*}$ such that $w$ lies in the kernel of $\overline{\phi}$, but is not trivial in $\overline{\BS}(2,3)$. Given any r.e.~set $W_{i}$, we form the recursive presentation $P_{i,j}:=\langle X|R\ (\cup S \textnormal{ if }j\in W_{i})\ \rangle$. That is, $P_{i,j}:=\langle X|R\rangle$ if $j\notin W_{i}$, and $P_{i,j}:=\langle X|R\cup S\rangle$ if $j\in W_{i}$. So $P_{i,j}$ is a recursive presentation (we add all the relators $S$ to $Q$ if we see $j\in W_{i}$). Now use our assumed algorithm that solves the word problem in $P_{i,j}$ to test if $\overline{\phi}(w)=e$ in $\overline{P}_{i,j}$; this will occur if and only if $j\in W_{i}$. By taking $W_{i}$ to be non-recursive, we derive a contradiction.
\end{proof}

\begin{note}
The above proof is not the original proof as developed by the author, but a simplified version due to Chuck Miller which follows from lemma \ref{jack} of this work. The author is grateful to Chuck for this simplification.
\end{note}

\begin{lem}\label{pre main}
There is a finite presentation $P=\langle X|R \rangle$ of a group with solvable word problem (namely, $\BS(2,3)$) for which there is no partial algorithm that, on input of a recursive presentation $Q=\langle Y|S \rangle$ such that $\overline{P} \cong \overline{Q}$, outputs a set map $\phi: X \to Y^{*}$ which extends to an isomorphism $\overline{\phi} : \overline{P} \to \overline{Q}$.
\end{lem}

\begin{proof}
Suppose such an algorithm exists. Given a recursive presentation $Q=\langle Y|S \rangle$ such that $\overline{P} \cong \overline{Q}$, use our supposed algorithm to output a set map $\phi: X \to Y^{*}$ which extends to an isomorphism $\overline{\phi} : \overline{P} \to \overline{Q}$. But since we have a solution to the word problem for $P$, we can combine this with the map $\phi$ to get a solution to the word problem for $Q$, thus contradicting theorem \ref{thm pre main}.
\end{proof}

Note that, by lemma \ref{truncate}, in the above proof there will always be a finite subset $S'$ of $S$ such that all other relators are consequences of $S'$, and hence $\overline{\langle Y|S \rangle} \cong \overline{\langle Y|S' \rangle}$. So we have the following immediate corollary:

\begin{cor}\label{find m}
There is a finite presentation $P=\langle X|R \rangle$ of a group with solvable word problem (namely, $\BS(2,3)$) for which there is no partial algorithm that, on input of a recursive presentation $Q=\langle Y|S \rangle$ such that $\overline{P} \cong \overline{Q}$, outputs a finite subset $S' \subseteq S$ such that all other relators $S \setminus S'$ are consequences of $S'$. Equivalently, construction of a finite truncation point $m$ of $S$ (with $m$ as in the proof of proposition $\ref{almost}$) is not algorithmically possible in general.
\end{cor}

As an aside, if we were instead to view isomorphisms between groups as Tietze transformations rather than set maps (see \cite[$\S 1.5$]{MKS} for a full exposition of this concept) then we can interpret lemma \ref{pre main} in the following way.

\begin{lem}\label{Ti version}
There is a finite presentation $P$ of a group with solvable word problem (namely, $\BS(2,3)$) for which there is no algorithm that, on input of a recursive presentation $Q$ such that $\overline{P} \cong \overline{Q}$, outputs a recursive enumeration of Tietze transformations of type $(\ti 1)$ and $(\ti 2)$ (manipulation of relators), and a finite set of Tietze transformations of type $(\ti 3)$ and $(\ti 4)$ (addition and removal of generators), with notation as per \cite[$\S 1.5$]{MKS}, transforming $P$ to $Q$.
\end{lem}

\begin{rem}
It should be pointed out that such a sequence of Tietze transformations as described above will always exist, which follows directly from lemma \ref{truncate}; we merely truncate the relators of $Q$ to get a finite presentation $Q'$, perform a finite sequence of Tietze transformations which takes $P$ to $Q'$ (possible by \cite[Corollary 1.5]{MKS}), and then add the rest of the enumeration of relators of $Q$ to $Q'$. The point here is that we cannot compute an enumeration of such a sequence in general.
\end{rem}

Before proceeding to our main result, we note the following lemma which shows that the direction in which we define our isomorphism in lemma \ref{pre main} is inconsequential.

\begin{lem}
Using the notation of lemma $\ref{pre main}$, having an explicit set map $\phi: X \to Y^{*}$ which extends to an isomorphism $\overline{\phi}$ allows us to construct a set map $\psi: Y \to X^{*}$ which extends to the inverse of $\overline{\phi}$ (and hence is an isomorphism). The reverse also holds. 
\end{lem}

\begin{proof}
This follows from the fact that we only deal with recursive presentations, and we can uniformly enumerate all trivial words of such presentations by lemma \ref{words}.
\end{proof}

We now have all the technical machinery required to prove our main result:

\begin{thm}\label{main}
There is a finitely presented group with unsolvable subgroup identification problem.
\end{thm}

Note that, by theorem \ref{can do for locally hopf}, such a group cannot be locally Hopfian.

\begin{proof}
Take a recursive enumeration $P_{i}:=\langle X_{i}|R_{i} \rangle$ of all recursive presentations of groups (fix some countably infinite alphabet $\mathcal{A}$; each $X_{i}$ is a finite subset of $\mathcal{A}$, and each $R_{i}$ is a recursive enumeration of words in $X_{i}^{*}$). By the Higman embedding theorem (see \cite[Theorem 12.18]{Rot}) we can embed their free product (with presentation $P_{1}*P_{2}*\ldots$) into a finitely presented group with presentation $P=\langle X|R \rangle$. Then the group $\overline{P}$ does not have solvable subgroup identification problem, for if it did then we could use this to contradict lemma \ref{pre main} as follows:
\\By the uniformity of Higman's result, as described in the proof of \cite[Theorem 12.18]{Rot}, there is a uniform procedure that, for any $i$, outputs a finite set of words $S_{i} \subset X^{*}$ in 1-1 correspondence with $X_{i}$ such that the subgroup $\langle S_{i} \rangle^{\overline{P}}$ is isomorphic to $\overline{P}_{i}$, and an explicit bijection $\phi_{i}: X_{i} \to S_{i}$ which extends to an isomorphism $\overline{\phi_{i}}: \overline{P_{i}} \to \langle S_{i} \rangle^{\overline{P}}$. That is, we can keep track of where each $\overline{P}_{i}$ is sent in this embedding. So, given a recursive presentation $Q=\langle Y|S\rangle$ and a finite presentation $H=\langle Z|V\rangle$ such that $\overline{H} \cong \overline{Q}$, compute $j$ such that $P_{j}=Q$ as recursive presentations (that is, number all Turing machines which read alphabet $Y$, and look for one identical to the description $S$). We know that $\overline{H}\cong \langle S_{j} \rangle^{\overline{P}}$. So use our algorithm to output a set map $\psi: Z \to S_{j}^{*}$ which extends to an isomorphism $\overline{\psi}: \overline{H} \to \langle S_{j} \rangle^{\overline{P}}$. But we can compose $\psi$ with $\phi_{j}^{-1}$ to get the set map $\phi_{j}^{-1} \circ \psi: Z \to Y^{*}$ which extends to an isomorphism $\overline{\phi_{j}^{-1} \circ \psi}: \overline{H} \to \overline{Q}$. But this is impossible by lemma \ref{pre main}, so our assumed algorithm can't exist.
\end{proof}

\section{Further work}

The group $G$ constructed in theorem \ref{main} contains an embedded copy of every recursively presented group, and so is not coherent, let alone weakly effectively coherent. We ask the question of whether theorem \ref{main} can be modified so that $G$ is weakly effectively coherent, or even merely coherent. Either of these would help make the result even more relevant, as we conjecture that strong and weak effective coherence are not equivalent properties for finitely presented groups.

\begin{tt}
\ 
\\Dipartimento di Matematica `Federigo Enriques'
\\Universit\`{a} degli Studi di Milano
\\Via Cesare Saldini 50, Milano, 20133, ITALIA
\\maurice.chiodo@unimi.it
\end{tt}

\end{document}